\documentclass[reqno]{amsart}
\usepackage{latexsym,amsmath,amssymb,amscd}
\usepackage[all]{xy}
\usepackage{enumerate}
\usepackage{hyperref}

\makeatletter
\@addtoreset{figure}{section}
\def\thefigure{\thesection.\@arabic\c@figure}
\def\fps@figure{h,t}
\@addtoreset{table}{bsection}

\def\thetable{\thesection.\@arabic\c@table}
\def\fps@table{h, t}
\@addtoreset{equation}{section}

\makeatother

\newtheorem{theorem}{Theorem}
\newtheorem{corollary}[theorem]{Corollary}

\newtheorem{lemma}[theorem]{Lemma}
\newtheorem{notation}[theorem]{Notation}

\newtheorem{proposition}[theorem]{Proposition}
\newtheorem{remark}[theorem]{Remark}

\numberwithin{theorem}{section}
\numberwithin{equation}{section}

\renewcommand{\1}{{\bf 1}}

\newcommand{\co}{{\rm co}}
\newcommand{\de}{{\rm d}}

\newcommand{\ext}{{\rm ext}}
\newcommand{\Hom}{{\rm Hom}}
\newcommand{\ie}{{\rm i}}

\newcommand{\Ker}{{\rm Ker}\,}

\newcommand{\supp}{{\rm supp}\,}

\newcommand{\tr}{{\rm tr}}

\newcommand{\CC}{{\mathbb C}}

\newcommand{\TT}{{\mathbb T}}

\newcommand{\Bc}{{\mathcal B}}
\newcommand{\Cc}{{\mathcal C}}

\newcommand{\Hc}{{\mathcal H}}
\newcommand{\Jc}{{\mathcal J}}

\title[Traces of $C^*$-algebras of  solvable groups]{Traces of $C^*$-algebras of\\ connected solvable groups}

\author{Ingrid Belti\c t\u a}
\author{Daniel Belti\c t\u a}
\address{Institute of Mathematics ``Simion Stoilow'' of the Romanian Academy,
P.O. Box 1-764, Bucharest, Romania}

\email{Ingrid.Beltita@imar.ro, ingrid.beltita@gmail.com}
\email{Daniel.Beltita@imar.ro, beltita@gmail.com}

\begin{document}

\parskip=5pt

\begin{abstract}
	We give an explicit description of the tracial state simplex of the $C^*$-algebra $C^*(G)$ of an arbitrary connected, second countable, locally compact, solvable group $G$. We show that every tracial state of $C^*(G)$ lifts from a tracial state of the $C^*$-algebra of the abelianized group, and  the intersection  of the kernels of all the tracial states of $C^*(G)$ is a proper ideal unless $G$ is abelian. 
	As a consequence, the $C^*$-algebra of a connected solvable nonabelian Lie group  cannot embed into a simple unital AF-algebra. 
\end{abstract}

\maketitle

\section{Introduction}

If $G$ is a locally compact group, then any of its Haar measures leads via integration to a multiplicative linear functional on  $C^*(G)$, 
which is in particular a tracial state. 
Moreover, 
the following assertions are equivalent: 
\begin{itemize}
\item $G$ is amenable; 
\item the reduced group $C^*$-algebra $C^*_r(G)$ is nuclear and admits a tracial state. 
\end{itemize}
See \cite[Th. 8]{Ng15} and also \cite[Cor. 3.3]{FSW17} and \cite{KR17}.
There arises the natural question of finding an explicit description of the set of all tracial states of group $C^*$-algebras of amenable locally compact groups. 
In this paper we answer that question for connected, second countable, locally compact, solvable groups, which includes all connected solvable Lie groups 
(Theorem~\ref{P6}\eqref{P6_item1bis}).
To this end we prove that the extreme points of the tracial simplex of 
group $C^*$-algebras  
correspond to the 1-dimensional group representations (Proposition~\ref{P5}). 
This solves the above problem via the Kre\u\i n-Milman theorem.

As an application, we show that 
if the $C^*$-algebra of such a group admits a faithful tracial state, then the group under consideration is necessarily abelian 
(Corollary~\ref{P6_item3}). 
This implies that the $C^*$-algebra of a connected solvable \emph{nonabelian} Lie group never embeds into a \emph{simple unital} AF-algebra (Corollary~\ref{simple_QD_emb}). 
In addition to our earlier examples of amenable connected Lie groups whose group $C^*$-algebras are not AF-embeddable \cite{BB18}, this is still another aspect of AF-embeddability theory of group $C^*$-algebras which points out the sharp contrast between the behaviours of the  $C^*$-algebras of discrete countable groups and connected Lie groups. 
We recall that the $C^*$-algebra of every amenable discrete countable group embeds into a simple unital AF-algebra, namely, the universal UHF-algebra \cite[Th. B]{Sc00}.

\section{Main result} 

\begin{notation}
\normalfont 
For an arbitrary $C^*$-algebra $A$ with its topological dual space $A^*$ we will use the following notation: 
\begin{itemize}
	\item $A^*_+:=\{\varphi\in A^*\mid 0\le \varphi\}$; 
	\item $A^*_{+,\tr}:=\{\varphi\in A^*_+\mid (\forall a,b\in A)\ \varphi(ab)=\varphi(ba)\}$; 
	\item $T(A):=\{\tau\in A^*_{+,\tr}\mid\Vert\tau\Vert=1\}$, 
	$T_{\le 1}(A):=\{\tau\in A^*_{+,\tr}\mid\Vert\tau\Vert\le 1\}$; 
	\item $\ext\, S:=\{\chi\in S\mid \chi\text{ is an extreme point of } S\}$ for any subset $S\subseteq A^*$; 
	\item $\overline{\co}\, S$ stands for the weak$^*$-closure of the convex hull of any subset  $S\subseteq A^*$. 
\end{itemize}
We also note the equality 
\begin{equation}\label{scalars}
T_{\le 1}(A)=\{t\varphi\mid t\in[0,1],\ \varphi\in T(A)\}.
\end{equation}
The elements of $T(A)$ are called \emph{tracial states} of $A$, and the sets $T(A)\subseteq T_{\le 1}(A)$ are regarded as a topological subspaces of~$A^*$ with respect to the weak$^*$ topology of~$A^*$. 
For every $\tau\in \ext\, T_{\le 1}(A)$ with $\Vert\tau\Vert\ne 0$ 
we necessarily have $\Vert \tau\Vert=1$, hence $\tau\in T(A)$, 
and then  $\tau\in  T(A)\cap \ext\, T_{\le 1}(A)\subseteq \ext\, T(A)$. 
Thus 
$\ext\,(T_{\le 1}(A))\subseteq \{0\}\cup \ext\, (T(A))\subseteq T_{\le 1}(A)$
hence, taking $\overline{\co}(\cdot)$ in these inclusions, we obtain 
\begin{equation}\label{Krein-Milman}
T_{\le 1}(A)
=\overline{\co}(\{0\}\cup \ext\, T(A))
\end{equation} 
since the set $T_{\le 1}(A)$ is compact  and convex, therefore 
$T_{\le 1}(A)=\overline{\co}(\ext\,T_{\le 1}(A))$ 
by the Kre\u\i n-Milman theorem. 
(In fact, by a slightly more elaborate argument, 
$\ext\,T_{\le 1}(A)=\{0\}\cup \ext\, T(A)$
by \cite[Prop. 6.8.7(ii)]{Di64},  since the points of $\ext\,T(A)$ are the finite characters of norm~1 in the sense of \cite[6.7.1]{Di64}.)

We also note for later use in the proof of Theorem~\ref{P6}\eqref{P6_item1} that for any subset $S\subseteq A^*$ we have 
\begin{equation}\label{kernels}
\bigcap\limits_{\varphi\in \overline{\co}\, (\{0\}\cup S)}\Ker\varphi
=\bigcap\limits_{\varphi\in S}\Ker\varphi
=\bigcap\limits_{t\in[0,1]}\bigcap\limits_{\varphi\in S}\Ker(t\varphi)
=\bigcap\limits_{\varphi\in \bigcup_{t\in[0,1]}tS}\Ker\varphi
\end{equation} 

For any $\varphi\in A^*_+$ we denote by $(\pi_\varphi,\Hc_\varphi,\xi_\varphi)$ the output of its corresponding GNS construction; thus $\pi_\varphi\colon A\to\Bc(\Hc_\varphi)$ is a $*$-representation with a cyclic vector $\xi_\varphi\in\Hc_\varphi$ satisfying 
\begin{equation}\label{GNS}
(\forall a\in A)\quad \varphi(a)=(\pi_\varphi(a)\xi_\varphi\mid\xi_\varphi).
\end{equation}
The notation $J\trianglelefteq A$ indicates that $J$ is a closed two-sided ideal of the $C^*$-algebra~$A$. 
\end{notation}

We prove the following lemma for completeness, although its assertions are mostly known, as indicated in Remark~\ref{characters} below. 

\begin{lemma}\label{L1}
For every $\varphi\in A^*_{+,\tr}$ the following assertions hold: 
\begin{enumerate}[{\rm(i)}]
	\item\label{L1_item1} 
	One has $A^+\cap \Ker\varphi=A^+\cap \Ker\pi_\varphi$. 
	\item\label{L1_item2} 
	The functional $\tau_\varphi\colon \pi_\varphi(A)''\to\CC$, $\tau_\varphi(T):=(T\xi_\varphi\mid\xi_\varphi)$, is a faithful normal tracial positive functional hence the von Neumann algebra $\pi_\varphi(A)''$ is finite. 
	\item\label{L1_item3} 
	If $\varphi\in\ext\, T(A)$, then the von Neumann algebra $\pi_\varphi(A)''$ is a factor. 
\end{enumerate}
\end{lemma}

\begin{proof}
\eqref{L1_item1} 
It folows by \cite[Cor. 2.4.10]{Di64} that $\Ker\pi_\varphi$ is the largest closed two-sided ideal of $A$ contained in $\Ker\varphi$.

To prove that $A^+\cap \Ker\varphi\subseteq\Ker\pi_\varphi$ let $a\in A^+$ arbitrary with $\varphi(a)=0$. 
For every $b\in A$ one has by the Schwarz inequality 
$\vert\varphi(a^{1/2}b)\vert^2\le\varphi(a)\varphi(b^*b)=0$, hence $a^{1/2}b\in\Ker \varphi$. 
Then for arbitrary $c\in A$ one has 
$$\Vert \pi_\varphi(a^{1/4})\pi_\varphi(c)\xi_\varphi\Vert^2
=(\pi_\varphi(a^{1/2})\pi_\varphi(c)\xi_\varphi\mid \pi_\varphi(c)\xi_\varphi)
=\varphi(c^*a^{1/2}c)=\varphi(a^{1/2}cc^*)
=0$$
by the above remark, with $b=cc^*$. 
Thus $\pi_\varphi(a^{1/4})\pi_\varphi(c)\xi_\varphi=0$ for all $c\in A$. 
Since $\pi_\varphi(A)\xi_\varphi$ is dense in $\Hc_\varphi$, 
we then obtain $\pi_\varphi(a^{1/4})=0$, hence $\pi_\varphi(a)=0$. 

\eqref{L1_item2} 
It is clear that $\tau_\varphi$ is a normal positive functional.  
The set $\pi_\varphi(A)$ is dense in $\pi_\varphi(A)''$ with respect to the strong operator topology (by the bicommutant theorem).
Therefore, 
in order to see that $\tau_\varphi$ is a tracial functional, it suffices to show that $\tau_\varphi\vert_{\pi_\varphi(A)}$ is tracial. 
The later property follows by $\varphi\in A^*_{+,\tr}$ since 
for all $a\in A$ one has $\tau_\varphi(\pi_\varphi(a))=\varphi(a)$ 
by \eqref{GNS}. 

The fact that 
$\tau_\varphi$ is faithful follows as in the proof of \cite[Lemma 4.2]{BB18}: 
Let $T\in\pi_\varphi(A)''$ 
with $\tau_\varphi(T^*T)=0$, that is,  $T\xi_\varphi=0$. 
Then for every $a\in A$, 
$$\Vert T\pi_\varphi(a)\xi_\varphi\Vert^2
=(\pi_\varphi(a)^*T^*T\pi_\varphi(a)\xi_\varphi \mid \xi_\varphi)
=\tau_\varphi(\pi_\varphi(a)^*T^*T\pi_\varphi(a)).$$
Since  $\tau_\varphi$ is a tracial state of $\pi_\varphi(A)''$, it follows that 
$$\Vert T\pi_\varphi(a)\xi_\varphi\Vert^2
=\tau_\varphi(\pi_\varphi(a)\pi_\varphi(a)^*T^*T)
=(\pi_\varphi(a)\pi_\varphi(a)^*T^*T \xi_\varphi \mid \xi_\varphi)=0,$$
Thus $T\pi_\varphi(a)\xi_\varphi=0$ for every $a\in A$,  hence $T=0$,

\eqref{L1_item3} 
For every central projection 
$p=p^*=p^2\in \pi_\varphi(A)''\cap\pi_\varphi(A)'$ 
we define 
$$\varphi_p(\cdot)
:=(\pi_\varphi(\cdot)p\xi_\varphi\mid p\xi_\varphi)
=(\pi_\varphi(\cdot)\xi_\varphi\mid p\xi_\varphi)$$
so that $\varphi=\varphi_p+\varphi_{\1-p}$ and $\varphi_p,\varphi_{\1-p}\in A^*_+$, where $\1\in\Bc(\Hc_\varphi)$ is the identity operator. 
Moreover, since $p\in \pi_\varphi(A)''$ and $\pi_\varphi(A)$ is dense in $\pi_\varphi(A)''$ with respect to the strong operator topology, 
there exists a net $\{c_i\}_{i\in I}$ in $A$ with 
$p=\lim\limits_{i\in I}\pi_\varphi(c_i)$ in the strong operator topology in $\Bc(\Hc_\varphi)$. 
Then for arbitrary $a,b\in A$ we obtain
\begin{align*}
\varphi_p(ab)
&=(\pi_\varphi(ab)p\xi_\varphi\mid \xi_\varphi)
=(\pi_\varphi(a)p\pi_\varphi(b)\xi_\varphi\mid \xi_\varphi)
=\lim\limits_{i\in I}(\pi_\varphi(ac_ib)\xi_\varphi\mid \xi_\varphi) \\
&=\lim\limits_{i\in I}\varphi(ac_ib)
=\lim\limits_{i\in I}\varphi(c_iba)
=\lim\limits_{i\in I}(\pi_\varphi(c_iba)\xi_\varphi\mid \xi_\varphi)
=(p\pi_\varphi(ba)\xi_\varphi\mid \xi_\varphi) \\
&=\varphi_p(ba), 
\end{align*}
by \eqref{GNS} and the fact that $\varphi$ is a tracial functional.
Thus $\varphi_p\in A^*_{+,\tr}$ 
and similarly $\varphi_{\1-p}\in A^*_{+,\tr}$. 
If $p\xi_\varphi\ne0$ and $(\1-p)\xi_\varphi\ne0$, then 
$\psi_0:=\frac{1}{\Vert p\xi_\varphi\Vert^2}\varphi_p\in T(A)$ 
and  $\psi_1:=\frac{1}{\Vert (\1-p)\xi_\varphi\Vert^2}\varphi_{\1-p}\in T(A)$ 
and moreover $\varphi=\Vert p\xi_\varphi\Vert^2\psi_0+\Vert (\1-p)\xi_\varphi\Vert^2\psi_1$, where $\Vert p\xi_\varphi\Vert^2+\Vert (\1-p)\xi_\varphi\Vert^2=\Vert\xi_\varphi\Vert^2=\Vert \varphi\Vert=1$, 
hence we obtain a contradiction with the hypothesis $\varphi\in\ext\, T(A)$. 

Therefore one must have either $p\xi_\varphi=0$ or $(\1-p)\xi_\varphi=0$, 
that is, either $p\xi_\varphi=\xi_\varphi$ or $p\xi_\varphi=0$.  
Then, using the fact that $\pi_\varphi(A)\xi_\varphi$ is dense in $\Hc_\varphi$ and $p\in\pi_\varphi(A)'$, we obtain either $p=\1$ or $p=0$. 
Since $p$ is an arbitrary orthogonal projection in the center of $\pi_\varphi(A)''$, it follows that $\pi_\varphi(A)''$ is a factor. 
\end{proof}

\begin{remark}\label{characters}
\normalfont
Lemma~\ref{L1}\eqref{L1_item2} can also be derived from \cite[Prop. 6.8.3]{Di64}. 
As noted  above,
 the points of $\ext\, T(A)$ are the finite characters of norm~1, therefore 
the conclusion of Lemma~\ref{L1}\eqref{L1_item3} can be obtained by \cite[Th. 6.7.3]{Di64}.
We also recall that the finite characters of norm~1 correspond bijectively to the quasi-equivalence classes of nontrivial finite-factor representations by \cite[Cor. 6.8.6]{Di64}. 
\end{remark}

	For an arbitrary locally compact group $G$ 
	we denote 
	\begin{equation}\label{defG1}
	G^{(1)}:= \text{ the closed subgroup generated by} \;\;  \{ghg^{-1}h^{-1}\mid g,h\in G\}.
	\end{equation}
	Then $G^{(1)}$ is a closed normal subgroup of $G$ and, by \cite[Prop. 8.C.8]{BkHa19},  there exists a natural surjective $*$-morphism 
	$$\theta\colon C^*(G)\to C^*(G/G^{(1)}).$$ 
	Hence, denoting 
\begin{equation}\label{Jc}	
\Jc_G:=\Ker\theta, 
\end{equation}
	we obtain the short exact sequence of $C^*$-algebras
	\begin{equation}\label{quotient}
	0\to\Jc_G\hookrightarrow C^*(G)\mathop{\longrightarrow}\limits^\theta 
	C^*(G/G^{(1)})\to 0.
	\end{equation}
	
	\begin{lemma}\label{R4}
		We have that
	\begin{equation}\label{commutative}
	\Jc_G=\{0\} \iff G\text{ is commutative}. 
	\end{equation} 
	\end{lemma}

\begin{proof}
	If the group $G$ is abelian then $G^{(1)}=\{\1\}$ and $\theta$ is a $*$-isomorphism, hence $\Jc_G=\{0\}$. 
	Conversely, if $\Jc_G=\{0\}$, then the short exact sequence~\eqref{quotient} shows that $\theta \colon C^*(G)\to C^*(G/G^{(1)})$ is a $*$-isomorphism. 
	Since the quotient group $G/G^{(1)}$ is abelian, its $C^*$-algebra is commutative. 
	It then follows that $C^*(G)$ is commutative, which in turn implies that its dense $*$-subalgebra $L^1(G)$ is commutative, and this further implies that the group $G$ is abelian. 
	(See \cite[A.3.1, page 321]{Ri60} and also \cite[Th. (20.24)]{HR63} for a more general result.)
\end{proof}

	We denote by $P_1(G)\subseteq L^\infty(G)$ the set of all positive-definite functions $f\colon G\to\CC$ with $f(\1)=1$, 
	and let 
	$$P_{1,c}(G):=\{f\in P_1(G)\mid (\forall x,y\in G)\quad f(xyx^{-1})=f(y)\}$$
	be the set of all \emph{central} positive-definite functions. 
	
	For every $f\in P_1(G)$, let $\varphi_f\colon C^*(G)\to\CC$  be its corresponding state of $C^*(G)$.
	By \cite[Th. 13.5.2]{Di64}, the affine mapping 
	\begin{equation}\label{affine}
	\Phi\colon P_1(G)\to S(C^*(G)), \quad f\mapsto\varphi_f
	\end{equation}
	is a homeomorphism when $P_1(G)$ is endowed with the topology of uniform convergence on the compact subsets of $G$, while the state space $S(C^*(G))$ is endowed with its weak$^*$-topology. 
	It is also well-known that 
	\begin{equation}\label{affine_char}
	\Phi(P_{1,c}(G))=T(C^*(G)).
	\end{equation} 
	(See e.g., \cite[18.1.3 and proof of Prop. 17.3.1]{Di64} and \cite[\S 1.1]{FSW17}.)
	
	\begin{remark}\normalfont
	 In the above setting, we note for later use that $\Hom(G,\TT)\subseteq P_{1,c}(G)$. 
	Thus if $\chi\in \Hom(G,\TT)$ 
	and if we denote by $\widetilde{\chi}\colon C^*(G)\to\CC$ the $*$-representation obtained by integrating the unitary representation $\chi$, then 
	for every $g\in L^1(G)\subseteq C^*(G)$ we have 
	$\widetilde{\chi}(g)=\int_G g(x)\chi(x)\de x=\varphi_\chi(g)$, 
	hence 
	\begin{equation}\label{char}
	\varphi_\chi=\widetilde{\chi} \text{ for every }\chi\in \Hom(G,\TT). 
	\end{equation} 
	\end{remark}

\begin{lemma}\label{L2}
Let $G$ be a connected, 
locally compact group and $A:=C^*(G)$. 
If $\varphi\in\ext\, T(A)$, then 
there exist a positive integer $n\ge 1$ 
such that $\pi_\varphi$ is weakly equivalent to an irreducible representation 
of $A$ on an $n$-dimensional Hilbert space. 
If moreover the group $G$ is solvable, then $n=\dim\Hc_\varphi=1$ and 
there exists $\chi\in\Hom(G,\TT)$ with  $\pi_\varphi=\widetilde{\chi}$. 
\end{lemma}

\begin{proof}
It follows by Lemma~\ref{L1} that $\pi_\varphi(A)''$ is a finite factor. 
On the other hand, since the representation $\pi_\varphi\colon A=C^*(G)\to\Bc(\Hc_\varphi)$ is cyclic, hence nondegenerate, 
it corresponds to a unitary representation of~$G$. 
Therefore, since the group $G$ is connected, it follows by \cite[Th. 1]{KS52} that $\pi_\varphi$ has no type~II$_1$ direct summand. 
Thus the finite factor $\pi_\varphi(A)''$ is not type~II$_1$, and then it is type I$_n$ for some positive integer $n$, that is, 
one has a $*$-isomorphism $\pi_\varphi(A)''\simeq M_n(\CC)$. 
Since $\pi_\varphi(A)$ is dense in $\pi_\varphi(A)''$ with respect to the strong operator topology, it follows that $\pi_\varphi(A)\simeq M_n(\CC)$.

Since $\pi_\varphi$ is a factor representation, there exists an irreducible $*$-representation $\chi\colon A=C^*(G)\to \Bc(\Hc_\chi)$ 
(see \cite[3.9.1(c) and 5.7.6((b),(d))]{Di64}) 
with $\Ker\pi_\varphi=\Ker\chi$. 
One then  has $*$-isomorphisms 
$$\chi(A)\simeq A/\Ker\chi=A/\Ker\pi_\varphi\simeq \pi_\varphi(A)\simeq M_n(\CC)$$
and it then  follows that $\dim\Hc_\chi=n$. 
Moreover, the irreducible $*$-representation $\chi$ is obtained by integrating a certain unitary irreducible representation of $G$, 
denoted also as $\chi\colon G\to U(n)\subseteq M_n(\CC)$. 

If moreover the group $G$ is solvable, we obtain $n=1$ by Sophus Lie's classical theorem. 
(See e.g., \cite[Part I, Ch. V, \S 5]{Se06}.) 
Therefore $\pi_\varphi(A)=\pi_\varphi(A)''=\CC \1$. 
Then, since $\xi_\varphi\in \Hc_\varphi$ is a cyclic vector for the representation $\pi_\varphi$ we obtain $\Hc_\varphi=\CC\xi_\varphi$. 
Thus $\dim\Hc_\varphi=1$, hence $\chi\in\Hom(G,\TT)$, and this completes the proof. 
\end{proof}

\begin{proposition}\label{P5}
If $G$ is a connected, 
 locally compact, solvable group, then the  map $\Phi\vert_{\Hom(G,\TT)}\colon \Hom(G,\TT)\to \ext \, T(C^*(G))$, $\chi\mapsto \varphi_\chi$ 
is a homeomorphism. 
\end{proposition}

\begin{proof}
Use Lemma~\ref{L2} and \eqref{affine}--\eqref{affine_char}.
\end{proof}

In the statement of the next theorem, we use the notation in \eqref{defG1} and \eqref{Jc}.

\begin{theorem}\label{P6}	
Let $G$ be  a connected, 
 locally compact, solvable group.
Then 
the following assertions hold: 
	\begin{enumerate}[{\rm (i)}]
		\item\label{P6_item1} The closed two-sided ideal $\Jc_G$ is the interesection of all kernels of tracial states of $C^*(G)$, that is, 
		$$
		\Jc_G
		=\bigcap\limits_{\varphi\in T(C^*(G))}\Ker\varphi. 
		$$
		\item\label{P6_item2}
		$\Jc_G$ 
		is equal to the closed linear span of any of the sets 
		$$\{ab-ba\mid a,b\in C^*(G)\} \quad \text{and} \quad \{a\in C^*(G)\mid a^2=0\}.$$
		\item\label{P6_item1bis} 
		The mapping $T(C^*(G/G^{(1)}))\to T(C^*(G))$, $\tau\mapsto \tau\circ\theta$, is bijective. 
		 \end{enumerate} 
\end{theorem}

\begin{proof}
	\eqref{P6_item1} We actually prove the following equality:
	\begin{equation}
	\label{P6_item1_eq1}
	\Jc_G
	=\bigcap\limits_{\chi\in \Hom(G,\TT)}\Ker\varphi_\chi 
	=\bigcap\limits_{\varphi\in T(C^*(G))}\Ker\varphi. 
	\end{equation} 
It follows by Proposition~\ref{P5}
and the consequence of the Kre\u\i n-Milman theorem~\eqref{Krein-Milman} along with \eqref{kernels} and \eqref{scalars} that 
\begin{align*}
\bigcap\limits_{\chi\in \Hom(G,\TT)}\Ker\varphi_\chi 
& =\bigcap\limits_{\varphi\in \ext\,T(C^*(G))}\Ker\varphi \\
& =\bigcap\limits_{\varphi\in \overline{\co}(\{0\}\cup \ext\,T(C^*(G)))}\Ker\varphi \\
& =\bigcap\limits_{\varphi\in T_{\le 1}(C^*(G))}\Ker\varphi \\
&=\bigcap\limits_{\varphi\in T(C^*(G))}\Ker\varphi  ,
\end{align*}
hence the second equality in~\eqref{P6_item1_eq1} holds true. 

We now prove the first equality in~\eqref{P6_item1_eq1}. 
Let $q\colon G\to G/G^{(1)}$ be  the canonical quotient map.
Then, 
since the group $G/G^{(1)}$ is abelian,  the map
$$q^*\colon  \widehat{G/G^{(1)}}\to\Hom(G,\TT),\quad \omega\mapsto \omega\circ q$$
is bijective and is actually a homeomorphism. 
Using the canonical homeomorphism
$\widehat{G}\simeq\widehat{C^*(G)}$ along with the short exact sequence~\eqref{quotient}, we then obtain 
$$\Jc_G=\bigcap_{\omega\in \widehat{G/G^{(1)}}}\Ker \widetilde{\omega\circ q}
=\bigcap_{\chi\in\Hom(G,\TT)}\Ker \widetilde{\chi}
=\bigcap_{\chi\in\Hom(G,\TT)}\Ker \varphi_\chi$$
where the later equality follows by \eqref{char}.

\eqref{P6_item2}	
Let us denote by $A_0$ the set of all norm-convergent series $\sum\limits_{n\ge1}(a_na_n^*-a_n^*a_n)$ for $a_1,a_2,\ldots\in A:=C^*(G)$.   
If we denote by $A_{00}$ the closed linear span of the set $\{ab-ba\mid a,b\in C^*(G)\}$, then 
it is clear that 
$$A_0+\ie A_0\subseteq A_{00}\subseteq 
\bigcap\limits_{\varphi\in T(A)}\Ker\varphi=A_0+\ie A_0$$
where the latter equality follows by \cite[Thms. 2.6 and 2.9]{CP79}. 
Therefore, the above inclusions are actually equalities and then, 
using \eqref{P6_item1_eq1}, 
we obtain $A_{00}=\Jc_G$. 
On the other hand, $A_{00}$ coincides with the closed linear span of $\{a\in C^*(G)\mid a^2=0\}$ by \cite[Prop. 2.2]{AEVBS}. 

\eqref{P6_item1bis} 
As $\theta$  is a surjective $*$-morphism, the mapping 
$T(C^*(G/G^{(1)}))\to T(C^*(G))$, $\tau\mapsto \tau\circ\theta$, is well defined and injective. 
To prove that this mapping is surjective, let $\varphi\in T(C^*(G))$ arbitrary. 
Then $\Ker\theta=\Jc_G\subseteq\Ker\varphi$ by \eqref{P6_item1}, hence the functional 
$\tau\colon C^*(G/G^{(1)})\to\CC$, $\theta(a)\mapsto \varphi(a)$, 
is well defined, and moreover $\tau\circ\theta=\varphi$. 
Since the group $G/G^{(1)}$ is abelian, we directly obtain $\tau\in T(C^*(G/G^{(1)}))$, and this completes the proof. 
\end{proof}

We are now in a position to prove the following generalization of \cite[Prop. 4.4]{BB18} to groups that need not be type~I.

\begin{corollary}\label{P6_item3}
If $G$ is a connected, second countable, locally compact, solvable group, 
then there exists a faithful $\varphi\in T(C^*(G))$ if and only if $G$ is abelian. 
\end{corollary}

\begin{proof}
If $G$ is abelian, then its dual group $\widehat{G}$ is second countable 
by \cite[Prop. A.G.3]{BkHa19}. 
Then we may select a dense sequence $\{\chi_n\}_{n\ge 1}$ in $\widehat{G}$. 
Denoting by $\delta_\chi$ the Dirac measure concentrated at any point $\chi\in\widehat{G}$, 
it then follows that $\mu:=\sum\limits_{n\ge 1}\frac{1}{2^n}\delta_{\chi_n}$ is a probability Radon measure on $\widehat{G}:=\Hom(G,\TT)$ with $\supp\mu=\widehat{G}$. 
We then define 
$$f\colon G\to\CC,\quad f(x):=\int\limits_{\widehat{G}}\chi(x)\de\mu(\chi).$$
Since $G$ is abelian, it is clear that $f\in P_{1,c}(G)$.
On the other hand the Fourier transform gives the $*$-isomorhism $C^*(G)\simeq \Cc_0(\widehat{G})$. 
Then, 
since $\supp\mu=\widehat{G}$, it is straightforward to check that $\varphi_f\in T(C^*(G))$ is faithful. 

Conversely, if there exists $\varphi\in  T(C^*(G))$ with $(C^*(G))^+\cap\Ker\varphi=\{0\}$ then, by~\eqref{P6_item1_eq1}, we obtain  
$(C^*(G))^+\cap\Jc_G=\{0\}$. 
Since $\Jc_G$ is a $C^*$-algebra, this implies $\Jc_G=\{0\}$ since $\Jc_G$ is linearly spanned by its positive cone $(C^*(G))^+\cap\Jc_G$. 
(See e.g., \cite[1.5.7(2)]{Di64}.)
Then Lemma~\ref{R4} shows that the group $G$ is abelian. 
Alternatively, $\Jc_G=\{0\}$ implies by Theorem~\ref{P6}\eqref{P6_item2} that $C^*(G)$ is commutative, and then the group $G$ is abelian as in the proof of Lemma~\ref{R4}. 
\end{proof}

\begin{corollary}\label{simple_QD_emb}
If $G$ is a connected, second countable, locally compact, solvable group, 
then the following assertions are equivalent: 
\begin{enumerate}[{\rm(i)}]
	\item\label{simple_QD_emb_item1} 
	$C^*(G)$ embeds as a closed $*$-subalgebra of a unital simple AF-algebra.
	\item\label{simple_QD_emb_item3} 
	$C^*(G)$ embeds as a closed $*$-subalgebra of a unital  quasidiagonal simple $C^*$-algebra. 
	\item\label{simple_QD_emb_item4} 
	$G$ is abelian. 
\end{enumerate}
\end{corollary}

\begin{proof}
%
%
\eqref{simple_QD_emb_item1}$\Rightarrow$\eqref{simple_QD_emb_item3}: 
Clear.

\eqref{simple_QD_emb_item3}$\Rightarrow$\eqref{simple_QD_emb_item4}: 
Let us assume that there exists a unital  quasidiagonal simple $C^*$-algebra $A$ with an embedding  $C^*(G)\subseteq A$. 
Since $A$ is unital quasidiagonal, it follows by \cite[Prop. V.4.2.7]{Bl06} that $A$ has a tracial state $\varphi\colon A\to\CC$. 
Since $A$ is simple, we have that  $\Ker\pi_\varphi=\{0\}$. 
Thus Lemma~\ref{L1}\eqref{L1_item1} implies $A^+\cap \Ker\varphi=\{0\}$, that is, $\varphi$ is faithful. 
Then $\varphi\vert_{C^*(G)}$ is a faithful positive tracial functional on $C^*(G)$, which is moreover nonzero since the $C^*$-algebra $C^*(G)$ is generated by its positive cone. 
Therefore $G$ is abelian by Corollary~\ref{P6_item3}. 

\eqref{simple_QD_emb_item4}$\Rightarrow$\eqref{simple_QD_emb_item1}: 
Since $G$ is abelian, there is  a $*$-isomorphism $C^*(G)\simeq\Cc_0(\widehat{G})$. 
Let $\overline{G}$ denote the 1-point compactification of $\widehat{G}$ if 
$\widehat{G}$ is not compact, and $\overline{G}=\widehat{G}$ otherwise. 
Since $G$ is second countable, its dual group $\widehat{G}$ is second countable 
by \cite[Prop. A.G.3]{BkHa19}, 
hence $\overline{G}$ is a compact metrizable space.  
Then there exists a continuous surjective mapping 
$K\to \overline{G}$, where $K$ is the Cantor set, hence one has embeddings
$$C^*(G)\simeq\Cc_0(\widehat{G})\hookrightarrow\Cc(\overline{G})\hookrightarrow\Cc(K).$$
On the other hand, it is well-known that $\Cc(K)$ is $*$-isomorphic to the canonical diagonal Cartan subalgebra of the CAR-algebra $M_{2^\infty}$. 
We thus finally obtain an embedding $C^*(G)\hookrightarrow M_{2^\infty}$,
where $M_{2^\infty}$ is a unital simple AF-algebra. 
\end{proof}

\begin{remark}
\normalfont
The $C^*$-algebras of many connected solvable Lie groups are AF-embeddable as shown in \cite{BB18} and \cite{BB20}, hence Corollary~\ref{simple_QD_emb} shows that such embeddings cannot be done into unital simple AF-algebras in the case of nonabelian groups. 
\end{remark}

\begin{remark}
	\normalfont
	If $G$ is a connected, separable, locally compact, solvable group, using the notation of \cite[Def. 3.4.1]{Br06}, one has 
	\begin{equation}\label{lfd}
	T(C^*(G))={\rm UAT}(C^*(G))_{\rm LFD}
	\end{equation}
	that is, every tracial state of $C^*(G)$ is uniform locally finite dimensional, and in particular is quasidiagonal. 
	(See \cite[\S 3.5]{Br06}.)
	
	In fact, for arbitrary $\varphi\in T(C^*(G))$, there exists $\tau\in T(C^*(G/G^{(1)}))$ with $\varphi=\tau\circ\theta$ by Theorem~\ref{P6}\eqref{P6_item1bis}. 
	The quotient group $G/G^{(1)}$ is an abelian Lie group, 
	hence  $C^*(G/G^{(1)})$ is type~I, and then $\tau\in {\rm UAT}(C^*(G/G^{(1)}))_{\rm LFD}$ by \cite[Cor. 4.4.4]{Br06}. 
	Now, by \cite[Prop. 3.5.6]{Br06}, we obtain $\varphi\in{\rm UAT}(C^*(G))_{\rm LFD}$. 
	This proves the inclusion $\supseteq$ in \eqref{lfd}, while the converse inclusion is obvious. 
	
	However, we recall that there exist connected solvable Lie groups whose $C^*$-algebras are not quasidiagonal. (See \cite[Thm.~2.15]{BB18}.)
\end{remark}

\noindent 
\textbf{Acknowledgment}. 
We wish to thank the Referee for correcting some errors in the first version of our manuscript.

\end{document}